\documentclass{article}
\usepackage{arxiv}

\usepackage[utf8]{inputenc} 
\usepackage[T1]{fontenc}    
\usepackage{hyperref}       
\usepackage{url}            
\usepackage{booktabs}       
\usepackage{amsfonts}       
\usepackage{nicefrac}       
\usepackage[english]{babel}
\usepackage{ae}
\usepackage{eucal}
\usepackage[dvips]{graphicx}
\usepackage{epsfig}
\usepackage{graphicx}
\usepackage{amsfonts,amsthm}
\usepackage{amssymb}
\usepackage{a4}
\usepackage{apu}
\usepackage{xcolor}
\usepackage{amsmath}
\usepackage{mathtools}
\usepackage{multicol}

\usepackage{dcolumn,amsthm}

\renewenvironment{proof}[1][Proof]{\noindent\textit{#1. } }{\hfill$\square$}

\newtheoremstyle{theorem}{6pt}{6pt}{\rm}{}{\sffamily}{ }{ }{}
\theoremstyle{theorem}
\newtheorem{theorem}{\sc Theorem}[section]

\newtheoremstyle{lemma}{6pt}{6pt}{\rm}{}{\sffamily}{ }{ }{}
\theoremstyle{lemma}
\newtheorem{lemma}{\sc Lemma}[section]

\newtheoremstyle{example}{6pt}{6pt}{\rm}{}{\sffamily}{ }{ }{}
\theoremstyle{example}

\newtheoremstyle{corollary}{6pt}{6pt}{\rm}{}{\sffamily}{ }{ }{}
\theoremstyle{corollary}

\newtheoremstyle{definition}{6pt}{6pt}{\rm}{}{\sffamily}{ }{ }{}
\theoremstyle{definition}
\newtheorem{definition}[theorem]{\sc Definition}

\newtheoremstyle{remark}{6pt}{6pt}{\rm}{}{\sffamily}{ }{ }{}
\theoremstyle{remark}

\newtheoremstyle{approximation}{6pt}{6pt}{\rm}{}{\sffamily}{ }{ }{}
\theoremstyle{approximation}

\newtheoremstyle{scheme}{6pt}{6pt}{\rm}{}{\sffamily}{ }{ }{}
\theoremstyle{scheme}

\usepackage{biblatex}
\usepackage{csquotes}

\title{On the numerical computation of Killing and conformally Killing vector fields on compact Riemannian manifolds}

\author{
  Gaëlle Brunet \\
  Department of Physics and Mathematics\\
  University of Eastern  Finland\\
  P.O. Box 111, FI-80101 Joensuu,Finland  \\
  \texttt{gaelpa@uef.fi} \\
   \And
 Maryam Samavaki \\
  Department of Physics and Mathematics\\
  University of Eastern  Finland\\
  P.O. Box 111, FI-80101 Joensuu,Finland  \\
  \texttt{maryam.samavaki@uef.fi} \\
  \AND
  Jukka Tuomela \\
  Department of Physics and Mathematics\\
  University of Eastern  Finland\\
  P.O. Box 111, FI-80101 Joensuu,Finland  \\
  \texttt{jukka.tuomela@uef.fi} \\
}

\begin{document}
\maketitle

\begin{abstract}
The defining equations for Killing vector fields and conformal Killing vector fields are overdetermined systems of PDE. This makes it difficult to solve the systems numerically. We propose an approach which reduces the computation to the solution of a symmetric eigenvalue problem. The eigenvalue problem is then solved by finite element techniques. The formulation itself is valid in any dimension and for arbitrary compact Riemannian manifolds. The numerical results which validate the method are given in two dimensional case.
\end{abstract}


\section{Introduction}
\label{sec;introduction}
Killing vector fields, whose flows generate the isometries on Riemannian manifolds, are fundamental in differential geometry. They arise also indirectly in the study of geodesics. One way to approach the problem is to consider the geodesic flow on the cotangent bundle. Then one tries to find some quantities which are invariant by this flow. If such invariants can be found then the flow is "integrable", i.e. it allows a more explicit description. Darboux in his classic \cite[Chapitre II]{darboux} studied extensively this problem. Apparently he did not explicitly define Killing vector fields, but it turns out that finding a certain type of invariant to the geodesic flow is the same as  finding a Killing field on the manifold.

The Killing fields also appear in continuum mechanics because Killing fields are stationary solutions of incompressible Euler and Navier-Stokes equations. This has potentially important consequences when one studies atmospheric models. Let us consider the 2 dimensional Navier-Stokes equations on the sphere. In the absence of external forces the solution tends to some Killing field and not to zero \cite{NSE}. This phenomenon cannot be observed in the standard setting because boundary conditions do not allow the existence of Killing fields. Now the numerical methods for Navier-Stokes equations typically can add some dissipation to stabilize the computations. However, in the case of the sphere this can have the unintended consequence of dissipating the underlying Killing field. Hence the energy content of the solution is not correct and this can in turn have significant effects on the computed solutions. We will explore this aspect more thoroughly in a forthcoming paper. 

Of course in more complete models of atmospheric flows there are more equations than just the horizontal flow described by the Navier-Stokes equations. However, this horizontal flow is anyway always an important component of the whole model, and hence its analysis will be helpful in  understanding the properties of more complicated systems. For more discussion and analysis of atmospheric flows we refer to
\cite{majda}. 

Not all manifolds have Killing fields; the existence of these fields is analyzed for example in \cite{refintro1}. The conclusion is that Killing fields cannot exist if the Ricci tensor is "too negative", and if they exist the space of Killing fields is finite dimensional. Hence as a PDE system the Killing equations are of finite type, i.e. there are only a finite number of free parameters in the general solution. Actually it is rather difficult in practice to determine if the given metric admits any Killing fields. For two dimensional case there is a classical criterion (given below) but higher dimensional cases are still subject to research \cite{Boris}.

Conformal Killing fields is a certain kind of generalization of Killing fields; in other words Killing fields are also conformally Killing, but there may be fields which are conformally Killing but not Killing. Here also the existence of conformally Killing fields depends on the Ricci tensor, but now the Ricci tensor does not have to be "so positive", which allows the existence of more fields. Again as a PDE system the conformal Killing equations are of finite type except in dimension two. In two dimensional case the conformal Killing equations correspond to the Cauchy Riemann equations, so that locally the solution space is infinite dimensional. However, for compact manifolds without boundary the solution space is still finite dimensional.  

Conformal Killing fields (which are not Killing) are in fact quite different from Killing fields as we will observe below in the examples, and the questions where conformal Killing fields arise are apparently of rather different nature than the problems related to Killing fields. As the name suggests, the conformal Killing fields  appear in the studies related to the conformal equivalences of Riemannian manifolds. Also in some questions of relativity theory conformal Killing equations appear \cite{blau}. 

Killing fields and conformal Killing fields can be considered as vector fields or covector fields whichever is more convenient. Below we will consider them as vector fields. The defining condition for these fields has also been generalized for other tensor fields \cite{conformal}. However, below we will only consider vector fields. 

Because Killing equations are of finite type, in principle the whole field is determined by the relevant data at one point. However, numerically it is not obvious how to propagate this initial data in a stable way to the whole manifold to obtain a description of the whole field. In fact we are not aware of any general numerical schemes for computing Killing and conformal Killing fields.  In this article we propose a method to compute the Killing and conformal Killing vector fields by reducing the problems to a symmetric eigenvalue problem. Killing and conformal Killing vectors then appear as the eigenspace corresponding to the zero eigenvalue of an elliptic operator.  Other eigenvalues are positive, and incidentally one  could ask if the fields corresponding to positive eigenvalues have any interesting geometric or physical interpretation.

In the numerical solution of the eigenvalue problem we have used standard finite element techniques, and we have used the program \textsf{FREEFEM++} \cite{Freefem} in our computations. In the case of Klein bottle the identifications of the coordinate domain boundaries are such that we had to program this case with \textsf{C++}. Our formulation gives a well posed problem in any dimension, but below we will give numerical results only in two dimensional case.

The paper is organized as follows. In section 2 we review the necessary background in Riemannian geometry and functional analysis. In section 3 we recall a few relevant properties of the Killing and conformal Killing equations. In section 4 we formulate our problem as an eigenvalue problem and show that the problem is well posed.  In section 5 we give the numerical results in two dimensional case which validate our method.

\section{Preliminaries and notation}
\label{sec;preliminary}

\subsection{Geometry}
\label{sec2.1}
Let us review some basic facts about Riemannian geometry \cite{hebey,petersen}. Let $M$ be a smooth manifold with or without boundary with  Riemannian metric $g$.  In coordinates we write the vector field $u$ as $u=u^k$ or $u=u^k\partial_{x_k}$ if it is convenient to indicate the particular system of coordinates. The Einstein summation convention is used where appropriate.  The  covariant derivative of $u$ is given by
\[
    \nabla u=u^k_{;j}=u^k_{,j}+\Gamma^k_{ij}u^i
\]
The semicolon is used for the covariant derivative and comma for the standard derivative. $\Gamma$ is the Christoffel symbol of the second kind.   The usual operators are then given by the formulas
\begin{align*}
      \mathsf{div}(u)=&\mathsf{tr}(\nabla u)=u^k_{;k}\\
 \mathsf{grad}(f)=&g^{ij}f_{;j}\\
 \Delta f=&\mathsf{div}(\mathsf{grad}(f))=g^{ij}f_{;ij}
\end{align*}
The divergence operator can extended to general tensors by the formula $\mathsf{div}(T)=\mathsf{tr}(\nabla T)$.

The metric $g$ induces an inner product for general tensors. For one forms we can simply write $g(\alfa,\beta)=g^{ij}\alfa_i\beta_j$.
In addition for this we need the inner product for tensors of type $(1,1)$. Let $A$ be of type $(1,1)$ and  let $A^\ast$ be its adjoint, i.e.
\begin{equation*}
g(Au,v)=g(u,A^\ast v)
\end{equation*}
for all vector fields $u$ and $v$. Then the inner product on the fibers can be defined by
\begin{equation}
g(A,B)=\mathsf{tr}(AB^\ast)=A^k_\ell g^{j\ell}B^i_j g_{ik}
      =A^{kj}B_{jk}
\label{sisätulo}
\end{equation}
The curvature tensor is denoted by $R$ and Ricci tensor by $\mathsf{Ri}$. There are several different conventions regarding the indices and signs of these tensors. We will follow the conventions in \cite{petersen}. In coordinates we have
\begin{equation}
            \mathsf{Ri}_{jk}=R^i_{ijk}
 \label{ricci-tensori}
 \end{equation}
In two dimensional case $\mathsf{Ri}=\kappa g$ where $\kappa$ is the Gaussian curvature.

Let $\partial M$ be the boundary of $M$. The divergence theorem is valid on Riemannian manifolds in the following form:
\begin{equation*}
     \int_M \mathsf{div}(u)\omega_M=
     \int_{\partial M} g(u,\nu )\omega_{\partial M}
\end{equation*}
where $\nu$ is the outer unit normal and $\omega_M$ is the volume form induced by the metric (or Riemannian density if $M$ is not  orientable) and $\omega_{\partial M}$ is the corresponding volume form or density on the boundary.

\subsection{PDE}
\label{sec2.2}

Let  $\alfa$ be a multiindex and $|\alfa|=\alfa_1+\dots+\alfa_n$. Then a general linear PDE can be written as
\begin{equation*}
Au=\sum_{|\alfa|\le q}b_\alfa\partial^\alfa u=f
\end{equation*}
where $b_\alfa$ are some known matrices, not necessarily square.
\begin{definition} The principal symbol of the operator $A$ is
\begin{equation*}
\sigma A=\sum_{|\alfa|= q}b_\alfa\xi^\alfa
\end{equation*}
$A$ is \emph{elliptic}, if $\sigma A$ is injective for all $\xi$ real and $\xi\ne 0$.
\label{elliptic}
\end{definition}
Let us from now on suppose that  $\sigma A$ is a square matrix because we will not need the more general case in the sequel. Let us then suppose that our PDE system is defined on some Riemannian manifold $M$. Then $\sigma A$ can be interpreted as a $(1,1)$ tensor, i.e. it defines a map $T_pM\to T_pM$.  The variables $\xi_i$ can then be interpreted as components of a one form. 

The characteristic polynomial of $\sigma A$ is $p_A(\xi)=\det(\sigma A)$. In this case $A$ is elliptic if $p_A\ne 0$ for all $\xi$ real and $\xi\ne 0$. It is clear that the order of $p_A$ must be even if $A$ is elliptic. It is known that for elliptic boundary value problems the number of boundary conditions must be half the order of the characteristic polynomial. When considering elliptic boundary value problems one needs to impose appropriate boundary conditions. The relevant criterion is known as Shapiro-Lopatinskij condition \cite{agranovich}.

\subsection{Functional analysis}
\label{sec2.3}
We will eventually formulate our problem as a spectral problem so let us recall the relevant theorem  which we will need. For details we refer to \cite{daulio3,hebey}.

Let us consider some Riemannian manifold $M$ and let us  define the inner product for vector fields by the formula
\begin{align*}
\langle u,v\rangle=\int_M g(u,v)\omega_M
\end{align*}
This gives the norm $\|u\|_{L^2}=\sqrt{ \langle u,u\rangle}$ and the corresponding space is denoted by $L^2(M)$.
In this way we can define the  Sobolev inner product
\begin{equation*}
  \langle u,v\rangle_{H^1}=\int_M \Big(g(u,v)+ g(\nabla u,\nabla v)\Big)\omega_M
\end{equation*}
where $g(\nabla u,\nabla v)$  is defined by the formula \eqref{sisätulo}. This gives the norm  $\|u\|_{H^1}=\sqrt{ \langle u,u\rangle_{H^1}}$ and the corresponding Sobolev space is denoted by $H^1(M)$.

Let $V$ be a real Hilbert space and let $a\,:\, V \times V\to\mathbb{R}$ be a continuous and symmetric bilinear map. Let $H$ be another Hilbert space such that $V\subset H$ with compact and dense injection. Let us consider the following eigenvalue problem: 
\begin{itemize}
    \item[] find $\lambda $ and $u\ne0$ such that
\[
   a(u,v)=\lambda \langle u,v\rangle_H
\]
for all $v\in V$. 
\end{itemize}
Due to symmetry the eigenvalues are real.
We say that $a$ is \emph{coercive} if there are two constants $\alpha > 0$ and $\mu \in \mathbb{R}$ such that
\begin{equation*}
a(v,v)+\mu \|v\|_H^2 \geq \alpha \|v\|_{V}^2\quad \forall v\in V
\end{equation*}
The result we will need is the following. 
\begin{theorem}\label{LM}
Let $a$ be symmetric, continuous and coercive. Then there are real numbers $\lambda_k$ and elements $u_k\in V$ such that
\begin{equation*}
a(u_k,v)=\lambda_k \langle u_k,v \rangle_H \quad,\ \forall v\in V
\end{equation*}
where $
-\mu< \lambda_1 \leq \lambda_2 \leq\dots $ and $\lambda_k\to \infty$ when $k\to\infty$. Moreover all eigenspaces are finite dimensional and they are orthogonal to each other with respect to the inner product of $H$.
\end{theorem}

\section{Killing and conformal Killing vector fields}

 Let $u$ be a vector field on $M$ and let us define the following operators:
\begin{align*}
Su=&g^{kj}u^i_{;j}+g^{ij}u^k_{;j}\\
Cu=&Su-\tfrac{2}{n}\,\mathsf{div}(u)g^{ij}
\end{align*}
\begin{definition} A vector field $u$ is a Killing vector field if $Su=0$ and a conformal Killing field if $Cu=0$.
\end{definition}
Let us first summarize some facts about the existence and non-existence of these fields. For more details we refer to \cite{refintro1,petersen,conformal}. 

As a PDE system $Su=0$ is a system of $\tfrac{1}{2}\,n(n+1)$ linear first order equations with $n$ unknown functions. Differentiating all equations we find that one can actually express all second order derivatives  in terms of lower order derivatives:
\begin{equation*}
    u^{i}_{;jk}=-u^{\ell}R^{i}_{\ell kj}
\end{equation*}
where $R$ is the curvature tensor. Using the language of formal theory of PDE \cite{pommaret,seiler} one can say that by prolonging the system once one gets an involutive form which is of finite type. Hence the dimension of the solution space is at most $\tfrac{1}{2}\,n(n+1)$, and  this upper bound is actually attained for $S^n$ and $\mathbb{R}^n$.

Also  $Cu=0$  is a system of $\tfrac{1}{2}\,n(n+1)$ equations with $n$ unknowns, but now the dimension 2 is a special case. When $n=2$ there are actually only 2 independent equations and it is easily seen that the resulting system is elliptic.\footnote{In $\mathbb{R}^2$ with standard metric one obtains Cauchy Riemann equations.} Hence locally the space of conformal Killing fields is infinite dimensional. When $n>2$  one has to prolong twice to see that the  the system is of finite type so that in this case the solution space is finite dimensional even locally. The relevant computations are carried out in \cite[p. 133]{pommaret}. For compact manifolds without boundary the dimension of the solution space is finite even for $n=2$.  The upper bound for the dimension of the solutions space is $\tfrac{1}{2}\,(n+1)(n+2)$ for $n>2$ in all cases and the same bound is valid when $n=2$ for compact manifolds wihtout boundary. Again this upper bound is attained for the standard spheres.

The existence of Killing and conformal Killing fields depends on the curvature. If the Ricci tensor is everywhere negative definite then there can be no Killing and conformal Killing fields on the manifolds without boundary. On the spheres Ricci tensor is always positive definite; hence the conditions for the existence are ''favorable'' and in this sense it is rather ''natural'' that the upper bound for the dimension is attained for the spheres.

Let us indicate one way of checking if there are any Killing fields on the surface for the given metric. Let us introduce the following covectors :
\begin{align*}
\beta &=\tfrac{1}{2}\,d\, g(d\kappa,d\kappa)=\kappa_{;i}g^{ij}\kappa_{;jk}\\
\alpha &= d\Delta \kappa= g^{ij} \kappa_{;ijk}
\end{align*}
where $\kappa$ is the Gaussian curvature. Then there is the following classical criterion \cite{Boris}.
\begin{lemma} 
\label{number_KF}
Let $M$ be a two dimensional Riemannian manifold and let $\kappa$ be the Gaussian curvature.
\begin{enumerate}
\item If $\kappa$ is constant then locally the space of Killing fields is three dimensional
\item If $\kappa$ is not constant and $d\kappa \otimes \beta$  and $d\kappa \otimes \alpha$ are symmetric   then locally the space of Killing fields is one dimensional
\item Otherwise there are no Killing fields.
\end{enumerate}
\end{lemma}

On the other hand if one would like to compute which metrics admit Killing fields then the symmetry conditions above give a system of two nonlinear PDE for the three components of the metric. One equation is of the fourth order and the other is of fifth order. Nonlinearity and high order makes this system difficult to handle even though the system is underdetermined.

\section{Eigenvalue problem}
We will from now on always suppose that $M$ is compact. 
Then let us write $S_u$ and $C_u$ when we consider $Su$ and $Cu$ as tensors of type $(1,1)$; pointwise they are thus maps $T_pM\to T_pM$. 
One can readily check that $S_u$ is symmetric: i.e. $g(S_u v,w)=g(v,S_u w)$ for all  $v$ and $w$. Obviously then $C_u$ is also symmetric.

Let us then introduce the following bilinear maps:
\begin{align*}
a_K\,:\,H^1(M)\times H^1(M)\to\mathbb{R}\quad,\quad
  a_K(u,v)=& \tfrac{1}{2} \int_M g(S_u,S_v)\omega_M \\
  a_C\,:\,H^1(M)\times H^1(M)\to\mathbb{R}\quad,\quad
  a_C(u,v)=& \tfrac{1}{2} \int_M g(C_u,C_v)\omega_M
\end{align*}
Then we can formulate the following eigenvalue problems:
\begin{itemize}
\item[(K)] Find $u\in H^1(M)$ and $\lambda$ such that
\begin{equation*}
a_K(u,v)=
   \lambda \int_M g(u,v)\omega_M
\end{equation*}
for all $v\in H^1(M)$.
\item[(CK)] Find $u\in H^1(M)$ and $\lambda$ such that
\begin{equation*}
a_C(u,v)=
   \lambda \int_M g(u,v)\omega_M
\end{equation*}
for all $v\in H^1(M)$.
\end{itemize}
Now evidently $a_K(u,u)\ge 0$ and  $a_C(u,u)\ge 0$ for all $u$, and   $a_K(u,u)=0$ (resp.  $a_C(u,u)=0$) only if $u$ is Killing (resp. conformally Killing) so that the eigenspace of zero eigenvalue is the space of Killing fields (resp. conformally Killing fields). 

It is clear that $a_K$ and $a_C$ are symmetric and continuous, so that in particular $\lambda$ must be real. Then we should show that the maps $a_K$ and $a_C$ are coercive. 
Now in fact  the coercivity of $a_K$ in $\mathbb{R}^n$ is a classical result known as \emph{Korn's inequality}. This inequality can also be extended to the Riemannian context \cite{chen,taylor1} and  the corresponding coercivity result is also valid for the map $a_C$ \cite{sergio}.   

The following Theorem gives the result when the manifold has no boundary. This is not a new result, but we think that the proof is interesting because it is simple and it shows the result for both $a_K$ and $a_C$ in the same way.  In the following proof we use several formulas which are computed in \cite{NSE} to which we refer for details.
\begin{theorem} The maps $a_K$ and $a_C$ are coercive, if $M$ has no boundary.
\label{koersiivinen}
\end{theorem}  
\begin{proof} Let us introduce the operators $ L_Ku=\mathsf{div}(Su)$ and $ L_Cu=\mathsf{div}(Cu)$. Then we compute
\begin{align*}
 \mathsf{div}(S_u v) =&\tfrac{1}{2} g(S_u,S_u)+ g(Lu,v)\\
   \mathsf{div}(C_u v) =     & \tfrac{1}{2} g(C_u,C_u)+ g(L_Cu,v)
\end{align*}
On the other hand
\begin{equation}
\begin{aligned}
   L_Ku=&\Delta_Bu+\mathsf{grad}(\mathsf{div}(u))+\mathsf{Ri}(u)\\
   L_Cu=&\Delta_B u+\big(1-\tfrac{2}{n}\big)\mathsf{grad}(\mathsf{div}(u))+\mathsf{Ri}(u)
\end{aligned}
\label{Lu-oper}
\end{equation}
where $\Delta_Bu= \mathsf{div}(g^{ij} u^k_{;i})=g^{ij} u^k_{;ij}$ is the Bochner Laplacian. Hence on the manifolds without boundary
\begin{equation}
\begin{aligned}
a_K(u,u)=&\int_M \Big( g(\nabla u,\nabla u)+\mathsf{div}(u)^2 -\mathsf{Ri}(u,u)\Big)\omega_M \\
\ge&\int_M \Big( g(\nabla u,\nabla u)+\big(1-\tfrac{2}{n}\big)\mathsf{div}(u)^2 -\mathsf{Ri}(u,u)\Big)\omega_M\\
=&a_C(u,u)\ge \int_M \Big( g(\nabla u,\nabla u) -\mathsf{Ri}(u,u)\Big)\omega_M
\end{aligned}
\label{killing-ominaisuus}
\end{equation}
Pointwise $\mathsf{Ri}$ can be interpreted as a linear map $T_pM\to T_pM$. Taking the operator norm at each point we can define $\mu=\max_{p\in M}\|\mathsf{Ri}\|$. Since $M$ is compact $\mu$ is finite. Hence
\[
   a_K(u,u)\ge a_C(u,u)\ge \int_M \Big( g(\nabla u,\nabla u) -\mu g(u,u)\Big)\omega_M \ge
   \alpha\|u\|^2_{H^1}-(\mu+\alpha)\|u\|^2_{L_2}
\]
if $0<\alpha\le 1$. 
\end{proof}

Note that from the formula \eqref{killing-ominaisuus} it follows that if $u$ is Killing then
\[
\int_M  g(\nabla u,\nabla u)\omega_M =
\int_M \mathsf{Ri}(u,u)\omega_M 
\]
and if $u$ is conformally Killing then
\[
\int_M  \Big( g(\nabla u,\nabla u)+\big(1-\tfrac{2}{n}\big)\mathsf{div}(u)^2 \Big)\omega_M =
\int_M \mathsf{Ri}(u,u)\omega_M 
\]
This shows directly that if $\mathsf{Ri}$ is negative definite there can be no Killing and conformally Killing  fields.

When the manifold has a boundary the proof is more difficult. Anyway the following results are valid:
\begin{itemize}
\item[] If $M$ is compact with Lipschitz boudary $\partial M$ then $a_K$ is coercive \cite{chen,taylor1} and 
 $a_C$ is coercive  for $n>2$ \cite{sergio}.
\end{itemize}
Our eigenvalue problems are thus well posed. For numerical purposes it is convenient to express $a_K$ and $a_C$ in a different form. Straightforward computations give the following formulas:
\begin{align*}
  a_K(u,v)=&\int_M \Big( g(\nabla u,\nabla v) + \mathsf{tr}(\nabla u \nabla v)  \Big)\omega_M \\
  a_C(u,v)=& \int_M \Big( g(\nabla u,\nabla v) + \mathsf{tr}(\nabla u \nabla v) - \tfrac{2}{n}\mathsf{div}(u)\mathsf{div}(v) \Big)\omega_M 
\end{align*}
It is perhaps useful to interpret the eigenvalue problems in the classical form. Let $p\in \partial M $ and let $\{\tau_1,\dots,\tau_{n-1}\}$ be a basis of $T_p\partial M$ and let $\nu$ be the outer unit normal vector.  Using the operators $L$ and $L_C$ introduced in the proof of Theorem \ref{koersiivinen} we can write the eigenvalue problems as follows. Again some details of the required computations can be found in \cite{NSE}.
 \begin{itemize}
\item[(K0)] Find $u$ and $\lambda$ such that
\[
\begin{cases}
-L_Ku=   \lambda u\\
 g(\nabla_\nu u,\tau_k) + g(\nabla_{\tau_k} u,\nu)=0\quad,\ k=1,\dots,n-1\\
   g(\nabla_\nu u,\nu) =0
\end{cases}
\]
\item[(CK0)] Find $u$ and $\lambda$ such that
\[
\begin{cases}
-L_Cu=   \lambda u\\
 g(\nabla_\nu u,\tau_k) + g(\nabla_{\tau_k} u,\nu)
 -\tfrac{2}{n}\,\mathsf{div}(u)g(v,\nu)=0\quad,\ k=1,\dots,n-1\\
   2g(\nabla_\nu u,\nu)  -\tfrac{2}{n}\,\mathsf{div}(u)g(v,\nu)=0
\end{cases}
\]
\end{itemize}
Note that if $u$ is Killing (resp. conformally Killing) then it satisfies the boundary conditions of problem (K0) (resp. problem (CK0)). Finally let us note that our operators are in fact elliptic. 
\begin{lemma}
Operators $L_K$ and $L_C$ are elliptic and moreover their symbols are symmetric.
\end{lemma}
\begin{proof}
Let us denote the identity map in $T_pM$ by $\mathsf{id}$.  From the formula \eqref{Lu-oper} it readily follows that 
\begin{equation*}
\sigma L_K=g(\xi,\xi)\mathsf{id}+g^{ij}\xi_j\xi_k
\end{equation*}
Then we compute
\begin{align*}
  g(\sigma L_K u,v)=&g(\xi,\xi)g(u,v)+\xi_i u^i\xi_jv^j= g(u, \sigma L_K v)\\
  g(\sigma L_K u,u)=&g(\xi,\xi)g(u,u)+(\xi_i u^i)^2
\end{align*}
Evidently the same computations prove the statement also for $L_C$.
\end{proof}

The well posedness of the eigenvalue problem thus also follows from the ellipticity of the operators $L_K$ and $L_C$. Note that the characteristic polynomials of $L_K$ and $L_C$ are of order $2n$ so that the number of the boundary conditions is correct in problems (K0) and (CK0).

\section{Numerical results}
\label{sec;examples}

\subsection{Implementation}

We have used standard finite element method in our computations, and almost everything was computed with the software \textsf{FREEFEM++} \cite{Freefem}. The standard algorithms produce a quasiuniform triangulation in Euclidean metric for the coordinate domain, but in our case it is important to modify this so that the resulting triangulation is quasiuniform in the given Riemannian metric. This can also be done with \textsf{FREEFEM++}. An example is shown in Figure \ref{metric} where on the left there is the initial triangulation and on the right is the adapted triangulation. The metric in this case corresponds to the standard torus which will be considered in the examples below.

\begin{figure}[!t]
   \centering
   \includegraphics[width=1\textwidth]{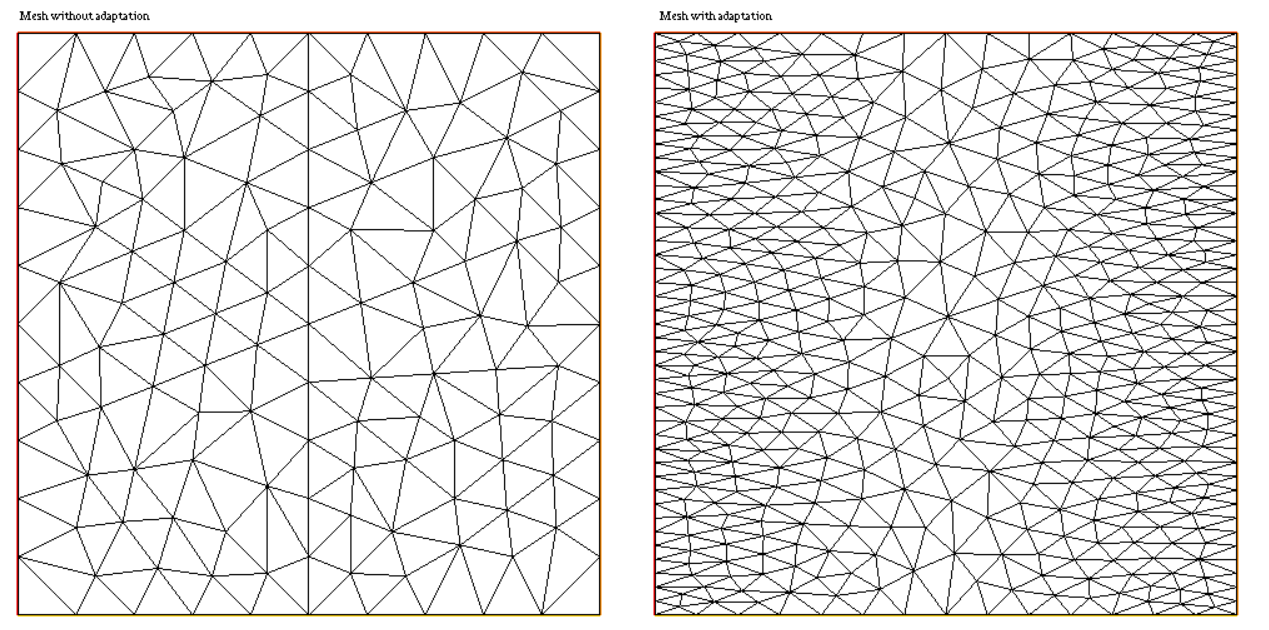}
   \caption{Initial triangulation on the left. The adapted triangulation on the right is quasiuniform for the appropriate Riemannian metric.}
   \label{metric}
\end{figure}

We will solve problems (K) and (CK) in three cases: Enneper's surface, torus and the Klein bottle. In case of Enneper's surface we have a manifold with boundary and a single coordinate chart so that the problem can be formulated in the standard way in \textsf{FREEFEM++}. The torus is a nontrivial manifold but analytically solving problems on the torus means that we look for the periodic solutions. Numerically this can be taken into account by so called periodic boundary conditions, and these are also implemented in \textsf{FREEFEM++}. 

The Klein bottle is a nonorientable surface which cannot be embedded in $\mathbb{R}^3$ but it can be embedded in $\mathbb{R}^4$. Here also one can use a single coordinate domain but now the identifications of the domain boundaries are nonstandard and cannot be done with \textsf{FREEFEM++}. In this case we implemented the appropriate identifications and the assembly of relevant matrices directly with \textsf{C++}.  

In all cases, for the numerical integration, we used quadrature formula on a triangle which is exact for polynomials of degrees less or equal to five. For more informations about the theory and implementation of quadratures, see \cite{ern}. We used \textsf{FREEFEM++} to visualize the computed solutions.

\subsection{Special properties of the two dimensional case}
In two dimensional case there is a special relationship between Killing and conformal Killing vector fields which is convenient to know when considering the examples. 
Let us introduce the tensor 
\[
   \eps=\sqrt{\det(g)}\big(dx_1\otimes dx_2-dx_2\otimes dx_1\big)
\]
Note that  $\nabla \eps=0$. Then let us define the operator $K$ by the formula
\begin{equation}
   v=Ku\quad\longleftrightarrow\quad v^k=g^{ki}\eps_{ij}u^j
\label{K-määr}
\end{equation}
Intuitively $K$ rotates the vector field by $90$ degrees. Then we have
\begin{lemma} Let $u$ be a Killing field. Then $Ku$ is a conformal Killing field. 
\label{conf-killing-rot}
\end{lemma}
\begin{proof} The Killing equations are
\begin{align*}
  & g^{11}u^1_{;1}+g^{12}u^1_{;2}=0\\
   &  g^{11}u^2_{;1}+g^{12}u^2_{;2}+g^{12}u^1_{;1}+g^{22}u^1_{;2}=0\\
   & g^{12}u^2_{;1}+g^{22}u^2_{;2}=0
\end{align*}
Let $v=Ku$; the conformal Killing equations are
\begin{align*}
    &g^{11}v^2_{;1}+g^{22}v^1_{;2}=0\\
    & g^{11}v^1_{;1}+2g^{12}v^1_{;2}-g^{11}v^2_{;2}=0
\end{align*}
Now simply substituting the covariant derivatives of $v$ to conformal Killing equations one checks that they are satisfied if $u$ satisfies the Killing equations.
\end{proof}

Note that this result shows that Killing vector fields and conformal Killing vector fields are of  completely  different nature, at least in two dimensional case. For example on the sphere Killing fields generate rotations so they give rise to Hamiltonian dynamics. Conformal Killing fields on the other hand describe the gradient dynamics. 

 A surface of revolution is a surface in $\mathbb{R}^3$ which has the parametrization
\[
  \varphi(x)=\begin{pmatrix}
  c_1(x_1)\cos(x_2) \\
  c_1(x_1)\sin(x_2) \\
 c_2(x_1)
 \end{pmatrix}
\]
The curve $c(x_1)=(c_1(x_1),c_2(x_1))$ is known as the \emph{profile curve}, the curves on the surface with $x_1$ constant are \emph{parallels} and the curves with $x_2$ constant are \emph{meridians}.
\begin{lemma} On the surfaces of revolution vector fields $b\partial_{x_2}$ where $b$ is constant are Killing fields. There are no other Killing fields unless the profile curve has a constant curvature. 
\label{killing-rev}
\end{lemma}
\begin{proof} The Killing equations for the surfaces of revolution are
\begin{align*}
  & |c'|^2u^1_{,1}+\langle c',c''\rangle u^1=0\\
   &   |c'|^2u^1_{,2}+c_1^2u^2_{,1}=0\\
   & c_1u^2_{,2}+c_1'u^1=0
\end{align*}
Clearly the fields $u=b\partial_{x_2}$ are solutions and by Lemma \ref{number_KF} there can be no other Killing fields, unless the profile curve has a constant curvature. 
\end{proof}

By Lemma \ref{conf-killing-rot} we thus have the following conformal Killing field on the surface of revolution:
\begin{equation}
    v=K\partial_{x_2}=g^{11}\eps_{12}\partial_{x_1}
    \label{conf-ratk}
\end{equation}

\subsection{Enneper's surface}
Our first example is the classical Enneper's surface which is also a minimal surface. Enneper's surface in $\mathbb{R}^3$ is given by the following map:
\[
  \varphi(x)=\begin{pmatrix}
  x_1 - \tfrac{1}{3}x_1^3 + x_1x_2^2 \\
 -x_2 + \tfrac{1}{3}x_2^3 - x_1^2x_2 \\
 x_1^2 - x_2^2
 \end{pmatrix}
\]
Let us recall that a coordinate system of a two dimensional Riemannian manifold is \emph{isothermal}, if the metric is of the form
\[
   g=e^{\lambda(x)}\big(dx_1\otimes dx_1+dx_2\otimes dx_2\big)
\]
for some function $\lambda$. The metric for Enneper's surface is of this form  with $\lambda =2\ln\big(1+|x|^2\big)$. 
Simply doing the computations we find that when the parametrization is isothermal then
\begin{equation}
   Su=0\quad\Longleftrightarrow\quad
   \begin{cases}
    u^1\lambda_{,1} + u^2\lambda_{,2} + 2u^2_{,2} = 0 \\
    u^1_{,2} + u^2_{,1} = 0  \\
    u^1_{,1} - u^2_{,2} = 0 \\
\end{cases}
\label{isothermal}
\end{equation}
where comma denotes the standard (not covariant) derivative. Note that the second and third equations are the Cauchy Riemann equations for components of $u$. 

In this case the Killing field can be explicitly computed. 
\begin{lemma} Vector fields $u=- b\,x_2 \partial_{x_1} + b\, x_1 \partial_{x_2}$ where $b$ is a constant are Killing fields on Enneper's surface.
\end{lemma}
Since the curvature is not constant (it is in fact $\kappa=-4/(1+|x|^2)^4$) there are no other Killing fields by Lemma \ref{number_KF}.

\begin{proof} The system \eqref{isothermal} gives in this case
\[
     \begin{cases}
  2 x_1 u^1 +2x_2 u^2 + u^2_{,2}(1+|x|^2) = 0 \\
    u^1_{,2} + u^2_{,1} = 0  \\
    u^1_{,1} - u^2_{,2} = 0 \\
\end{cases}
\]
This is equivalent to\footnote{The command \textsf{rifsimp} in {\sc Maple} is useful here.}
\[
     \begin{cases}
     x_1u^1+x_2 u^2=0\\
    u^2_{,2}=0\\
    x_1u^2_{,1}-u^2=0
    \end{cases}
\]
From this the result easily follows.
\end{proof}

Let us consider the coordinate domain $\mathcal{D}$ defined by the following boundaries :

\begin{align*}
\partial\mathcal{D}_1 &= \{ \big(\tfrac{1}{2}(\cos(t)+1),\sin(t)\big)\ ,\ t\in[0,\tfrac{\pi}{2}]  \} &
\partial\mathcal{D}_2 &= \{ \big(\tfrac{1}{2}-t, 1\big)\ ,\ t\in[0,\tfrac{1}{2}]   \} \\
\partial\mathcal{D}_3 &= \{ \big(\cos(t),\sin(t)\big)\ ,\ t\in[\tfrac{\pi}{2},\pi]  \} &
\partial\mathcal{D}_4 &= \{ \big(-\tfrac{1}{2}(\cos(t)+1),-\sin(t)\big)\ ,\ t\in[0,\tfrac{\pi}{2}]  \} \\
\partial\mathcal{D}_5 &= \{ \big( t,-1\big)\ ,\ t\in[-\tfrac{1}{2},0]  \} &
\partial\mathcal{D}_6 &= \{\big(\cos(t),\sin(t)\big)\ ,\ t\in[3\tfrac{\pi}{2},2\pi]  \} 
\end{align*}

The triangulated domain, with around 2000 triangles,  is shown in Figure \ref{fig:1} on the left. The metric was used to adapt the triangulation so that it is quasiuniform on the surface. The surface with the triangulation is shown in Figure \ref{fig:1} on the right.

\begin{figure}[!t]
   \centering
   \includegraphics[width=1\textwidth]{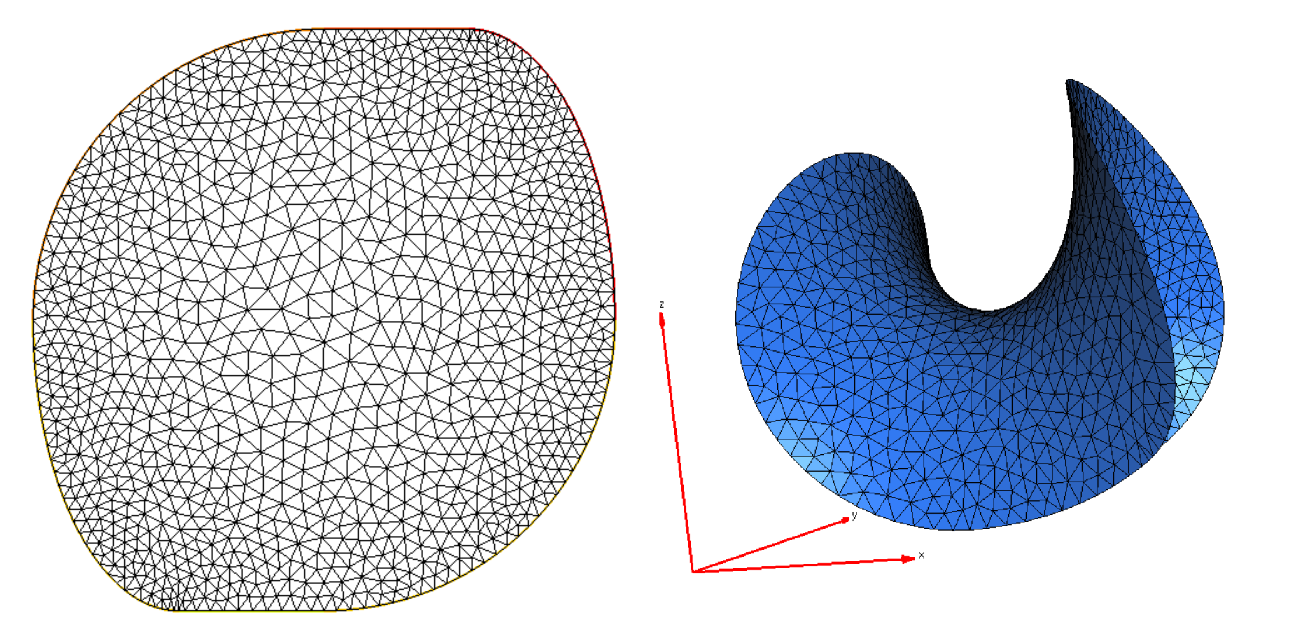}
   \caption{The coordinate domain of Enneper's surface on the left and the corresponding embedding in $\mathbb{R}^3$ on the right with adapted triangulation.}
   \label{fig:1}
\end{figure}

The computation of $a_K$ is in fact easy for all isothermal surfaces, and for Enneper's surface we obtain:
\[
a_k(u,v)=\int_\mathcal{D} ( 2u^1_{;1}v^1_{;1} + 2u^2_{;2}v^2_{;2} + u^1_{;2}v^1_{;2} + u^2_{;1}v^2_{;1} + u^1_{;2}v^2_{;1} + u^2_{;1}v^1_{;2} )(1+|x|^2)^2 dx_1dx_2
\]
Now in fact we  get essentially an exact solution up to rounding errors with $P_1$ elements. Checking the formulas for covariant derivatives one notices that if the components of $u$ and $v$ are polynomials of degree $m$ then the integrand is a polynomial of degree $2m+2$. Hence using $P_1$ elements integrands are of degree $4$ and they are integrated exactly by the default method of \textsf{FREEFEM++}. On the other hand analytically the components of exact solution are polynomials of degree one so that the approximation error is zero \cite{ern}. 

So already with about 100 triangles the approximate eigenvalue is $\lambda\approx 10^{-16}$ and the relative error in $L^2$ norm is about $10^{-13}$  and in $H^1$ norm it is about $10^{-12}$.
The computed field is shown  in Figure  \ref{fig:5}. 


\begin{figure}[!t]
   \centering
   \includegraphics[width=0.7\textwidth]{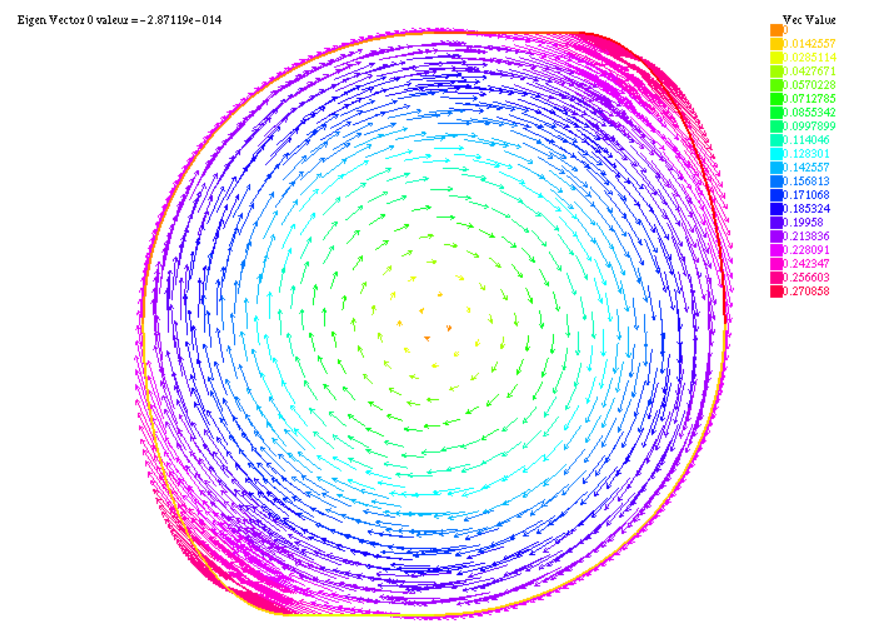}
    \caption{Representation of Killing fields on the Enneper's surface.}
    \label{fig:5}
\end{figure}

\subsection{Torus}
\label{sec7.2}
The flat torus is isothermal with $\lambda=0$ so that in this case the Killing fields are $u=b_1\partial_{x_1}+b_2\partial_{x_2}$. Hence in particular globally the space of Killing fields can be two dimensional although locally this is impossible by Lemma \ref{number_KF}. In this case one would also obtain exact solutions up to the rounding errors for the same reason as in the case of Enneper's surface.

Let us then consider the "standard" torus, with its Riemannian metric defined by the embedding in $\mathbb{R}^3$. This is a surface of revolution and as a profile curve we can choose
\[
c(x_1)=\big(2+\cos(x_1),\sin(x_1)\big)
\]
The corresponding metric  is given by
\begin{equation*}
g= dx_1 \otimes dx_1 + (2+ \cos(x_1))^2 dx_2\otimes dx_2
\end{equation*} 
By Lemma \ref{killing-rev} $u=\partial_{x_2}$ is a Killing field and by formula \ref{conf-ratk}
\[
    v=Ku=K\partial_{x_2}=\big(2+\cos(x_1)\big)\partial_{x_1} 
\]
is a conformal Killing field.

Note that a priori on a general surface of revolution there could be also other conformal Killing fields, but in  this particular case one can check that there are in fact no other conformal Killing fields. 

Our coordinate domain is thus the square  $[0,2\pi]\times [0,2\pi]$, with the boundaries appropriately identified. A representative solution for the conformal case, computed with around 2000 triangles, is shown in Figure \ref{fig:6}. The eigenspace corresponding to the zero eigenvalue is thus two dimensional, and it is spanned by a Killing field and a conformal Killing field which is not Killing. Numerically of course we have two eigenvalues very close to zero and each other.  Quantitative results are given below.


\begin{figure}[!t]
    \centering
    \includegraphics[width=1\textwidth]{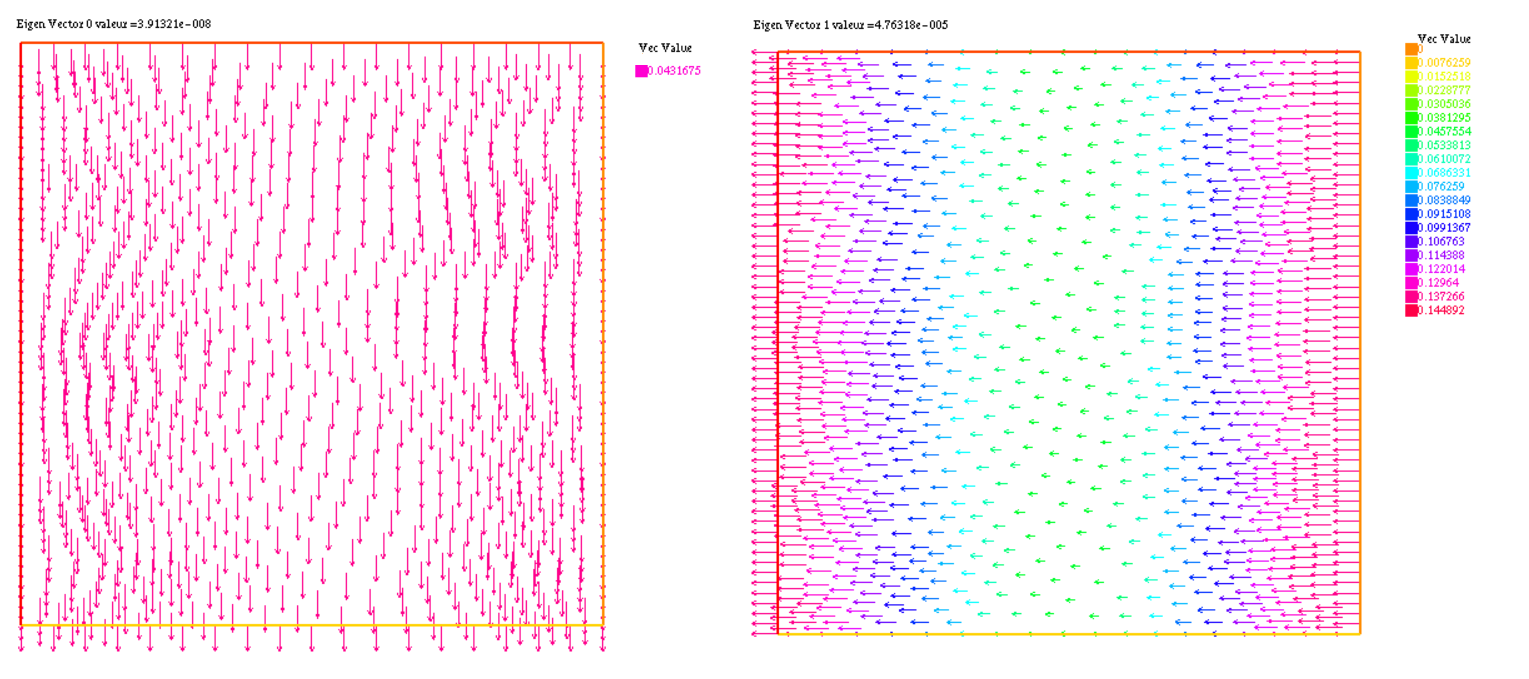}
     \caption{Approximation of a Killing field (on the left) and a conformal Killing field (on the right) on the standard torus.}
     \label{fig:6}
\end{figure}

\subsection{Klein bottle}
\label{sec6.3}
Let us finally consider the Klein bottle to see that our method works also on nonorientable surfaces. Klein bottle can be embedded in $\mathbb{R}^4$ and one popular parametrization is
\begin{equation}
      \varphi(x)=\begin{pmatrix}
 (2+\cos(x_1))\cos(x_2) \\
 (2+\cos(x_1))\sin(x_2) \\
\sin(x_1)\cos(x_2/2)\\
\sin(x_1)  \sin(x_2/2)
 \end{pmatrix}
 \label{klein-para}
\end{equation}

The parameter domain is again $[0,2\pi]\times [0,2\pi]$ and the sides are identified as in Figure   \ref{klein-bottle} on the left. The metric is
\[
   g=dx_1\otimes dx_1+\tfrac{1}{4}\big( 3\cos^2(x_1)+16\cos(x_1)+17\big)dx_2\otimes dx_2
\]
Locally this is like a surface of revolution which looks like (i.e. is isometric to) the surface shown in Figure   
 \ref{klein-bottle} on the right.
 
 Let $a=3\cos^2(x_1)+16\cos(x_1)+17$ and $b=\sin(x_1)(3\cos(x_1)+8)$. The Killing equations are now 
 \begin{align*}
  & u^1_{,1}=0\\
   &   4u^1_{,2}+a\,u^2_{,1}=0\\
   & a\,u^2_{,2}-b\,u^1=0
\end{align*}
and it is straightforward to check that $u=\partial_{x_2}$ is a solution. Then by Lemma \ref{conf-killing-rot}  $v=Ku=-\frac{\sqrt{3\cos(x)^2+16\cos(x)+17}}{2}\partial_{x_1}$ is a conformal Killing field. 

\begin{figure}
    \begin{center}
    \begin{tabular}{cc}
    \includegraphics[width=40mm]{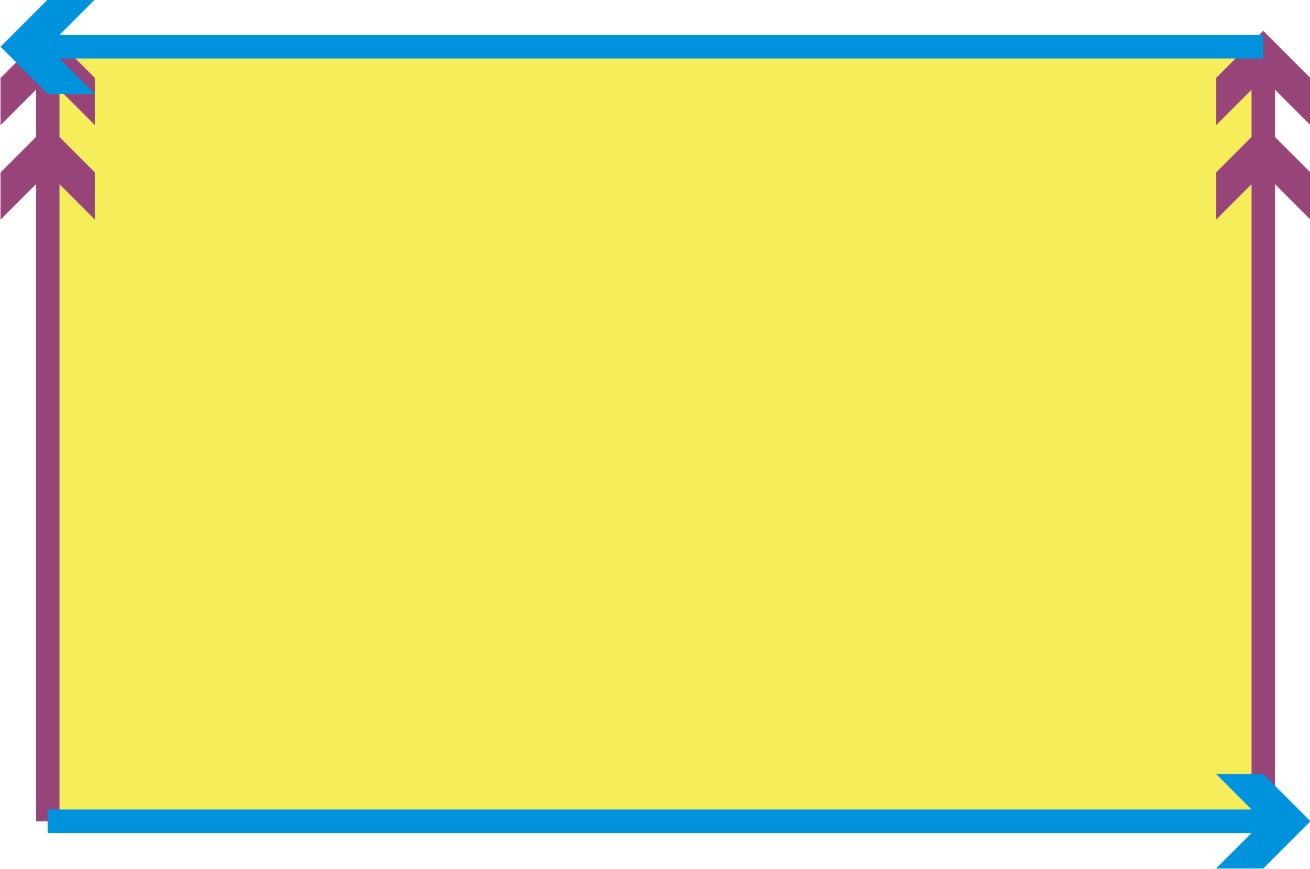}&
      \includegraphics[width=50mm]{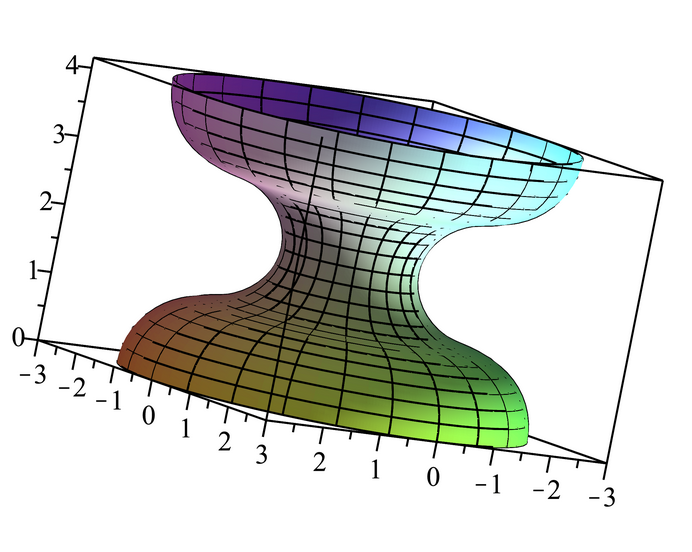}
      \end{tabular}
     \caption{Topological Klein bottle on the left and a surface of revolution which is locally  isometric to the Klein bottle with the embedding \eqref{klein-para}.}
     \end{center}
     \label{klein-bottle}
\end{figure}

The numerical results are discussed below.

\subsection{Numerical errors}

As explained above the case of Enneper's surface is rather special so let us here consider only the torus and the Klein bottle in more detail. We used $P_1$ elements for the Klein bottle and $P_2$ elements for the torus. 
In both cases we computed the solutions to problems (K) and (CK) with several triangulations. As we can see in tables \ref{table4} and \ref{table5}, even with few triangles (around 100), we have already quite  a  good approximation.  
In the conformal Killing case the eigenspace corresponding to zero eigenvalue is two dimensional. Numerically we have two eigenvalues very close to zero. Note that the approximation to the Killing field is much better than the approximation to the conformal Killing field. This is  because the components of the Killing field are simply constants so that the approximation error is zero and we see only the error arising in the numerical integration.

For completeness we computed the order of convergence in a standard way, i.e. we computed $k$ such that
\[
\varepsilon \approx Ch^k
\]
where $h$ is the maximum length of the triangulated domain and $\varepsilon $ is the error.
With 40 different triangulations for the example of the standard torus (\ref{sec7.2}), results are presented in the table \ref{table0}.
Results for the conformal Killing field (which are not Killing) are close to what one expects of  $P_2$ elements. For the Killing fields the convergence is faster because we see only the error due to numerical integration. The results are quite similar in the case of the Klein bottle. The convergence for conformal Killing fields are what one expects of $P_1$ elements, and again for Killing fields the order of the convergence is related to the order of numerical integration.

\begin{table}[!t]
\begin{center}
\begin{tabular}{|c|c|c|c|}
\hline
KF & $\varepsilon=|\lambda_h - \lambda|$ &  $\varepsilon=\|u_h - u\|_{L^2}$ & $\varepsilon=\|u_h - u\|_{H^1}$\\ 
\hline
k & 8.8 & 6.53 & 5.45 \\
\hline
CKF & $\varepsilon=|\lambda_h - \lambda|$ &  $\varepsilon=\|u_h - u\|_{L^2}$ & $\varepsilon=\|u_h - u\|_{H^1}$\\ 
\hline
k & 4.25 & 3.81 & 2.67 \\
\hline
\end{tabular}
\caption{The estimated order of convergence.}\label{table0}
\end{center}
\end{table}

\begin{table}[!t]
\begin{center}
\begin{tabular}{||c||c|c||}
\hline
 & \multicolumn{2}{c||}{\textbf{Torus}}\\
\hline
\textbf{100 triangles} & \textit{With adaptation} & \textit{Without adaptation} \\
\hline
Eigenvalue & $10^{-7}$ & $10^{-6}$ \\
\hline
$L^2$ norm of error  & $10^{-8}$ & $10^{-5}$ \\
\hline
$H^1$ norm of error  & $10^{-7}$ & $10^{-4}$  \\
\hline
\textbf{2 000 triangles} & \textit{With adaptation} & \textit{Without adaptation}  \\
\hline
Eigenvalue  & $10^{-10}$ & $10^{-7}$ \\
\hline
$L^2$ norm of error  & $10^{-14}$ & $10^{-10}$ \\
\hline
$H^1$ norm of error  & $10^{-14}$ & $10^{-8}$ \\
\hline
 & \multicolumn{2}{c||}{\textbf{Klein}}\\
\hline
\textbf{100 triangles} & \textit{With adaptation} & \textit{Without adaptation} \\
\hline
Eigenvalue & $10^{-7}$ & $10^{-4}$ \\
\hline
$L^2$ norm of error  & $10^{-6}$ & $10^{-4}$ \\
\hline
$H^1$ norm of error  & $10^{-5}$ & $10^{-3}$  \\
\hline
\textbf{2 000 triangles} & \textit{With adaptation} & \textit{Without adaptation}  \\
\hline
Eigenvalue  & $10^{-8}$ & $10^{-4}$ \\
\hline
$L^2$ norm of error  & $10^{-8}$ & $10^{-6}$ \\
\hline
$H^1$ norm of error  & $10^{-7}$ & $10^{-5}$ \\
\hline
\end{tabular}
\caption{Computations of Killing fields for the standard torus and Klein Bottle}
\label{table4}
\end{center}
\end{table}

\begin{table}[!t]
\begin{center}
\begin{tabular}{||c||c|c||}
\hline
 & \multicolumn{2}{c||}{\textbf{Torus}}\\
\hline
\textbf{100 triangles} & \textit{With adaptation} & \textit{Without adaptation}  \\
\hline
Eigenvalue & $10^{-4}$   & $10^{-3}$\\
\hline
$L^2$ norm of error  & $10^{-6}$ & $10^{-4}$ \\
\hline
$H^1$ norm of error  & $10^{-5}$ & $10^{-3}$  \\
\hline
\textbf{2 000 triangles} & \textit{With adaptation} & \textit{Without adaptation}  \\
\hline
Eigenvalue  & $10^{-7}$ & $10^{-4}$\\
\hline
$L^2$ norm of error  & $10^{-11}$ & $10^{-6}$ \\
\hline
$H^1$ norm of error  & $10^{-8}$ & $10^{-5}$ \\
\hline
& \multicolumn{2}{c||}{\textbf{Klein}}\\
\hline
\textbf{100 triangles} & \textit{With adaptation} & \textit{Without adaptation}  \\
\hline
Eigenvalue & $10^{-3}$   & $10^{-2}$\\
\hline
$L^2$ norm of error  & $10^{-3}$ & $10^{-2}$ \\
\hline
$H^1$ norm of error  & $10^{-2}$ & $10^{-1}$  \\
\hline
\textbf{2 000 triangles} & \textit{With adaptation} & \textit{Without adaptation}  \\
\hline
Eigenvalue  & $10^{-7}$ & $10^{-5}$\\
\hline
$L^2$ norm of error  & $10^{-7}$ & $10^{-3}$ \\
\hline
$H^1$ norm of error  & $10^{-6}$ & $10^{-3}$ \\
\hline
\end{tabular}
\caption{Computations of Conformal Killing fields which is not Killing for the standard torus and Klein Bottle}
\label{table5}
\end{center}
\end{table}

\begin{figure}[!t]
    \centering
    \includegraphics[width=0.8\textwidth]{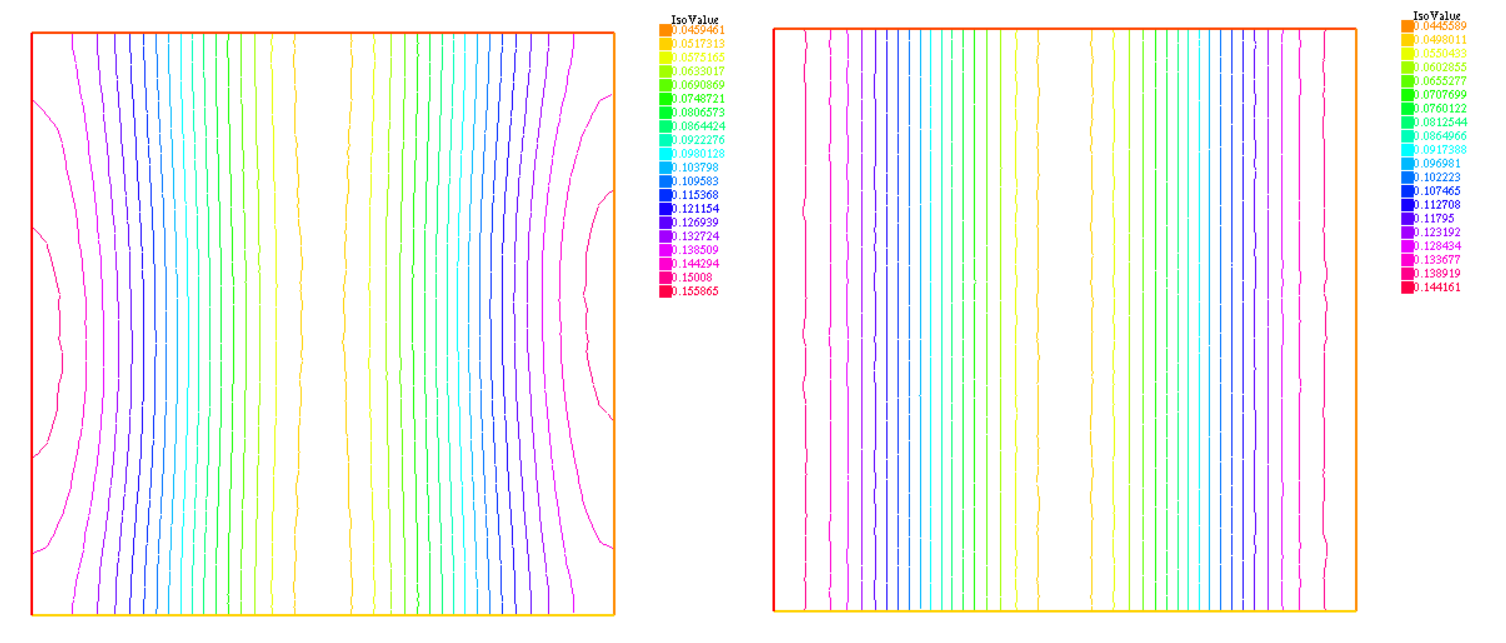}
     \caption{Component $u^1$ of a conformal Killing field which is not Killing on the Klein Bottle without (on the left) and with adaptation (on the right) of the metric.}
     \label{fig:7}
\end{figure}

The adaptation of the metric is important for computations as shown in Figure \ref{fig:7}. It represents the first component of the conformal Killing field which is not Killing on the Klein bottle. On the left, the domain is triangulated without adaptation, and it shows that the solution is deformed. That is not the case with the same number of triangles using an adapted mesh (on the right). It implies that $L^2$ and $H^1$ errors can increase significantly without an adapted mesh (see tables \ref{table4} and \ref{table5}). 

\printbibliography

\end{document}